\documentclass[11pt]{amsart}
\usepackage[margin=1in]{geometry}
\usepackage{hyperref,amssymb,amsfonts,amsthm,amsrefs,mathtools,enumerate,verbatim,
color,makecell}

\newtheorem{Theorem}{Theorem}[section]
\newtheorem{Proposition}[Theorem]{Proposition}
\newtheorem{Lemma}[Theorem]{Lemma}
\newtheorem{Corollary}[Theorem]{Corollary}

\newtheorem*{Claim}{Claim}

\theoremstyle{definition}
\newtheorem{Definition}[Theorem]{Definition}

\newcommand{\rca}{\mathsf{RCA}_0}
\newcommand{\wkl}{\mathsf{WKL}_0}
\newcommand{\aca}{\mathsf{ACA}_0}

\newcommand{\bso}{\mathsf{B}\Sigma^0_1}
\newcommand{\iso}{\mathsf{I}\Sigma^0_1}
\newcommand{\ipo}{\mathsf{I}\Pi^0_1}

\DeclareMathOperator{\bang}{{!}}

\newcommand{\rst}{{\restriction}}

\newcommand{\andd}{\wedge}
\newcommand{\orr}{\vee}
\newcommand{\la}{\langle}
\newcommand{\ra}{\rangle}

\newcommand{\imp}{\rightarrow}

\newcommand{\biimp}{\leftrightarrow}
\newcommand{\Biimp}{\Leftrightarrow}
\newcommand{\Nb}{\mathbb{N}}

\newcommand{\Rb}{\mathbb{R}}
\newcommand{\smf}{\smallfrown}

\newcommand{\leqKB}{\leq_\mathrm{KB}}
\newcommand{\leKB}{<_\mathrm{KB}}
\newcommand{\geKB}{>_\mathrm{KB}}
\newcommand{\geqKB}{\geq_\mathrm{KB}}

\newcommand{\Pf}{\mathcal{P}_\mathrm{f}}

\newcommand{\MP}[1]{\mathcal{#1}}

\usepackage{xcolor}	
\usepackage{soul}

\AtBeginDocument{%
   \def\MR#1{}
}

\title{The strength of compactness for countable complete linear orders}

\author{Paul Shafer}
\address{School of Mathematics\\
University of Leeds\\
Leeds\\
LS2 9JT\\
United Kingdom}
\email{p.e.shafer@leeds.ac.uk}
\urladdr{http://www1.maths.leeds.ac.uk/~matpsh/}

\date{\today}

\begin{document}
 
\begin{abstract}
We investigate the statement ``the order topology of every countable complete linear order is compact'' in the framework of reverse mathematics, and we find that the statement's strength depends on the precise formulation of compactness.  If we require that open covers must be uniformly expressible as unions of basic open sets, then the compactness of complete linear orders is equivalent to $\wkl$ over $\rca$.  If open covers need not be uniformly expressible as unions of basic open sets, then the compactness of complete linear orders is equivalent to $\aca$ over $\rca$.  This answers a question of Fran\c{c}ois Dorais.
\end{abstract}

\maketitle

\section{Introduction}

Every linear order $(L, \prec)$ can be equipped with its order topology, where the basic open sets are the open intervals
\begin{align*}
(a, \infty) &= \{x \in L : a \prec x\} \text{ for $a \in L$;}\\
(-\infty, b) &= \{x \in L : x \prec b\} \text{ for $b \in L$;}\\
(a, b) &= \{x \in L : a \prec x \prec b\} \text{ for $a, b \in L$ with $a \prec b$.}
\end{align*}

Call a linear order $(L, \prec)$ \emph{complete} if whenever $L$ is partitioned as $L = A^- \cup A^+$ with
\begin{align*}
(\forall x \in A^-)(\forall y \in A^+)(x \prec y),
\end{align*}
it is the case that either $A^-$ has a greatest element or $A^+$ has a least element.  It is well-known (see for example \cite[Theorem~27.1]{Munkres}) that the order topology of a non-empty linear order is compact if and only if the linear order is complete in the above sense.  Our goal is to characterize the logical strength of this fact when restricting to countable linear orders.

This work is an example of \emph{reverse mathematics}, which is the project of classifying mathematical theorems phrased in second-order arithmetic by the strengths of the axiom systems that are required to prove the theorems.  Reverse mathematics was introduced by H.\ Friedman~\cite{Friedman}, and the standard reference is Simpson's~\cite{SimpsonSOSOA}.  Formally, the only mathematical objects that second-order arithmetic allows are natural numbers and sets of natural numbers.  Nevertheless, straightforward coding techniques available in weak background theories allow us to discuss tuples and sequences of natural numbers; functions $f \colon \Nb^m \imp \Nb^n$; countable algebraic or combinatorial objects such as countable groups, rings, fields, graphs, trees, partial orders, linear orders; and more.  By coding a real number as a rapidly converging Cauchy sequence of rational numbers and by coding a basic open subset of $\Rb^n$ as an $(n+1)$-tuple of rational numbers (representing an open ball of rational radius whose center has rational coordinates), we may discuss $\Rb^n$, its topology, continuous functions $f \colon \Rb^m \imp \Rb^n$, and so forth.  We may even discuss arbitrary complete separable metric spaces by specifying a space's countable dense set and a metric on that dense set.  To date, most work in reverse mathematics involving topology has been confined to complete separable metric spaces.  Few attempts have been made to study general topology in second-order arithmetic.  One excellent example is Mummert's and Simpson's work on filter spaces~\cites{MummertMF, MummertSimpson}.  Here we take up Dorais's framework of \emph{countable second-countable topological spaces}~\cite{Dorais}.

Dorais's idea is to study general topology in second-order arithmetic by restricting to the topological spaces that can be straightforwardly represented in second-order arithmetic.  These are the countable second-countable spaces:  the topological spaces with countably many points and with countable bases.  The countable second-countable spaces framework of course comes with its own limitations (it is not a good approach to studying connectivity, for example), but it has the advantage of representing topological notions directly, and it works well for analyzing theorems whose proofs largely concern the combinatorics of open sets and closed sets.  This framework was used in~\cite{NoetherianSpaces} to analyze topological spaces arising from quasi-orders, for example.  Here we study countable linear orders and their order topologies.  This could be done in a completely \emph{ad hoc} manner, as a linear order's order topology is easy to describe.  However, the order topologies of countable linear orders fit very nicely into the countable second-countable spaces framework, as shown by the many examples in~\cite{Dorais}.

In~\cite{Dorais}, Dorais considers a notion of compactness in which the open sets of an open cover are explicitly presented as unions of basic open sets.  We call this notion \emph{compactness with respect to honest open covers}.  The idea is that a sequence of open sets is honest if it comes with an explanation of how each open set in the sequence can be written as a union of basic open sets.  Dorais observes that the base theory $\rca$ proves that if the order topology of a countable linear order is compact with respect to honest open covers, then that linear order is complete (see Lemma~\ref{lem-CompactImpComplete} below).  He then asks for the axiomatic strength of the converse, that is, for the strength of the statement ``the order topology of every countable complete linear order is compact with respect to honest open covers.''  Dorais shows that $\rca$ does not suffice to prove this statement.  In fact, he shows that, over $\rca$, the statement implies that there is no set of maximum Turing degree~\cite[Example~7.8]{Dorais}.  We answer Dorais's question by showing that the statement ``the order topology of every countable complete linear order is compact with respect to honest open covers'' is equivalent to $\wkl$ over $\rca$.

Dorais also considers a stronger notion of compactness, here simply called \emph{compactness}, where the open sets of an open cover need not be uniformly presentable as unions of basic open sets~\cite{DoraisPC}.  We show that with this notion of compactness, the statement ``the order topology of every countable complete linear order is compact'' is equivalent to $\aca$ over $\rca$.

\section{Preliminaries for working in second-order arithmetic}

We remind the reader of the axiom systems $\rca$, $\wkl$, and $\aca$.  Simpson's~\cite{SimpsonSOSOA} provides many more details concerning these and other systems, including many examples of theorems that can be proven in them.

The language of second-order arithmetic contains two sorts of variables:  first-order variables intended to range over the natural numbers and second-order variables intended to range over sets of natural numbers.  Typically, but not always, lower-case letters $a$, $b$, $c$, $x$, $y$, $z$, etc.\ denote first-order variables, and capital letters $A$, $B$, $C$, $X$, $Y$, $Z$, etc.\ denote second-order variables.  The symbol $\Nb$ is notational shorthand for the first-order part of whatever structure is implicitly under consideration.

The language of second-order arithmetic contains constant symbols $0$ and $1$, binary function symbols $+$ and $\times$, and binary relation symbols $=$, $<$, and $\in$.  The constants $0$ and $1$ name numbers, and the functions and relations $+$, $\times$, $=$, and $<$ only apply to numbers.  The relation $\in$ relates numbers to sets, and equality between sets is defined in terms of $\in$.

$\rca$ (standing for \emph{recursive comprehension axiom}) is an axiom system designed to capture computable mathematics.  Roughly speaking, to prove that some set exists when working in $\rca$, one must show how to compute that set, possibly using as an oracle some other set that has already been shown to exist.  The axioms of $\rca$ consist of
\begin{itemize}
\item a first-order sentence expressing that the numbers form a discretely ordered commutative semi-ring with identity; 

\smallskip

\item the \emph{$\Sigma^0_1$ induction scheme} (denoted $\iso$), which consists of the universal closures (by both number and set quantifiers) of all formulas of the form
\begin{align*}
[\varphi(0) \andd \forall n(\varphi(n) \imp \varphi(n+1))] \imp \forall n \varphi(n),
\end{align*}
where $\varphi$ is $\Sigma^0_1$; and 

\smallskip

\item the \emph{$\Delta^0_1$ comprehension scheme}, which consists of the universal closures (by both number and set quantifiers) of all formulas of the form
\begin{align*}
\forall n (\varphi(n) \biimp \psi(n)) \imp \exists X \forall n(n \in X \biimp \varphi(n)),
\end{align*}
where $\varphi$ is $\Sigma^0_1$, $\psi$ is $\Pi^0_1$, and $X$ is not free in $\varphi$.
\end{itemize}
$\rca$ is the usual base system or background theory in reverse mathematics.  Many theorems in reverse mathematics have the form $\rca \vdash \varphi \biimp \psi$, where $\varphi$ and $\psi$ are two well-known mathematical statements.  If $\rca \vdash \varphi \biimp \psi$, we say that $\varphi$ and $\psi$ are equivalent over $\rca$, and we interpret this as meaning that $\varphi$ and $\psi$ have equivalent logical strength.  The `$0$' in `$\rca$' refers to the restriction of the induction scheme to $\Sigma^0_1$ formulas.

$\rca$ proves several induction schemes, least element principles, bounding schemes, and bounded comprehension principles in addition to $\iso$.  Here the relevant schemes are the \emph{$\Pi^0_1$ induction scheme} (denoted $\ipo$), the \emph{$\Sigma^0_1$ least element principle}, the \emph{$\Sigma^0_1$ bounding scheme} (denoted $\bso$), and the \emph{bounded $\Sigma^0_1$ comprehension scheme}.  The $\Pi^0_1$ induction scheme is as the $\Sigma^0_1$ induction scheme, but the formula $\varphi$ is required to be $\Pi^0_1$.  The $\Sigma^0_1$ least element principle consists of the universal closures of all formulas of the form
\begin{align*}
\exists n \varphi(n) \imp \exists n[\varphi(n) \andd (\forall m < n)(\neg\varphi(m))],
\end{align*}
where $\varphi$ is $\Sigma^0_1$.  The $\Sigma^0_1$ bounding scheme consists of the universal closures of all formulas of the form
\begin{align*}
\forall a[(\forall n < a)(\exists m)\varphi(n,m) \imp \exists b(\forall n < a)(\exists m < b)\varphi(n,m)],
\end{align*}
where $\varphi$ is $\Sigma^0_1$ and $a$ and $b$ are not free in $\varphi$.  The $\Sigma^0_1$ bounded comprehension scheme consists of the universal closures of all formulas of the form
\begin{align*}
\forall b \exists X \forall n[n \in X \biimp (n < b \andd \varphi(n))],
\end{align*}
where $\varphi$ is $\Sigma^0_1$ and $X$ is not free in $\varphi$.  See~\cite[Section~I.2]{HajekPudlak} and \cite[Section~II.3]{SimpsonSOSOA} for further details.

$\rca$ suffices to code a finite set of numbers as a single number and a finite sequence of numbers as a single number in the usual way.  In $\rca$, we can thus code the set $\Nb^{<\Nb}$ of all finite sequences as well as its subset $2^{<\Nb}$ of all finite binary sequences.  We now fix our notation and terminology for (coded) sequences.  For $\sigma, \tau \in \Nb^{<\Nb}$, $|\sigma|$ denotes the length of $\sigma$, $\sigma \subseteq \tau$ denotes that $\sigma$ is an initial segment of $\tau$, and $\sigma^\smf\tau$ denotes the concatenation of $\sigma$ and $\tau$.  For $\sigma \in \Nb^{<\Nb}$ and $n \leq |\sigma|$, $\sigma \rst n = \la \sigma(0), \dots, \sigma(n-1) \ra$ denotes the initial segment of $\sigma$ of length $n$.  Likewise, if $f \colon \Nb \imp \Nb$ is a function and $n \in \Nb$, $f \rst n = \la f(0), \dots, f(n-1) \ra$ denotes the sequence consisting of the first $n$ values of $f$.

In $\rca$, we define a \emph{tree} to be a set $T \subseteq \Nb^{<\Nb}$ that is closed under initial segments:  $\forall \sigma \forall \tau [(\sigma \in T \andd \tau \subseteq \sigma) \imp \tau \in T]$.  A function $f \colon \Nb \imp \Nb$ is an \emph{infinite path} through a tree $T$ if every initial segment of $f$ is in $T$:  $\forall n (f \rst n \in T)$.  We can now define \emph{weak K\"onig's lemma} to be the statement ``every infinite subtree of $2^{<\Nb}$ has an infinite path.''  The system $\wkl$ is obtained by adding weak K\"onig's lemma to the axioms of $\rca$.

Lastly, the system $\aca$ (standing for \emph{arithmetical comprehension axiom}) is obtained by augmenting $\rca$ by the \emph{arithmetical comprehension scheme}, which consists of the universal closures of all formulas of the form
\begin{align*}
\exists X \forall n(n \in X \biimp \varphi(n)),
\end{align*}
where $\varphi$ is an arithmetical formula in which $X$ is not free.

$\wkl$ is strictly stronger than $\rca$, and $\aca$ is strictly stronger than $\wkl$ (see~\cite[Remark~I.10.2 and Section~VIII.2]{SimpsonSOSOA}).  A helpful characterization of $\aca$ is that, over $\rca$, it is equivalent to the statement ``every injection has a range.''
\begin{Lemma}[{\cite[Lemma~III.1.3]{SimpsonSOSOA}}]\label{lem-ACAinjection}
The following are equivalent over $\rca$.
\begin{enumerate}[(i)]
\item $\aca$.

\smallskip

\item If $f \colon \Nb \imp \Nb$ is an injection, then there is a set $X$ such that $\forall n(n \in X \biimp \exists s(f(s) = n))$.
\end{enumerate}
\end{Lemma}
Thus to show that some statement implies $\aca$ over $\rca$, it suffices to show that the statement implies that every injection has a range.

We make use of \emph{K\"onig's lemma} and \emph{bounded K\"onig's lemma} in addition to weak K\"onig's lemma.  Call a tree $T \subseteq \Nb^{<\Nb}$ \emph{finitely-branching} if every $\sigma \in T$ has at most finitely many immediate successors:  $(\forall \sigma \in T)(\exists n)(\forall m)[\sigma^\smf\la m \ra \in T \imp m < n]$.  Furthermore, call a finitely-branching tree $T \subseteq \Nb^{<\Nb}$ \emph{bounded} if it comes equipped with a function bounding its branching, i.e., if there is a function $g \colon \Nb \imp \Nb$ such that $(\forall \sigma \in T)(\forall n < |\sigma|)[\sigma(n) < g(n)]$.  K\"onig's lemma is the statement ``every infinite finitely-branching subtree of $\Nb^{<\Nb}$ has an infinite path,'' and bounded K\"onig's lemma is the statement ``every infinite bounded subtree of $\Nb^{<\Nb}$ has an infinite path.''

\begin{Theorem}\label{thm-KLequivs}{\ }
\begin{enumerate}[(i)]
\item\label{it-fullKL} K\"onig's lemma is equivalent to $\aca$ over $\rca$ (see \cite[Theorem~III.7.2]{SimpsonSOSOA}).

\smallskip

\item\label{it-bddKL} Bounded K\"onig's lemma is equivalent to $\wkl$ over $\rca$ (see \cite[Lemma~IV.1.4]{SimpsonSOSOA}).
\end{enumerate}
\end{Theorem}
 
\section{Countable second-countable topological spaces}

We introduce the countable second-countable topological spaces framework from~\cite{Dorais}.

\begin{Definition}[$\rca$; {\cite[Definition~2.1]{Dorais}}]
A \emph{strong base} (or simply \emph{base}) for a topology on a set $X$ is an indexed sequence $\MP U = (U_i)_{i \in I}$ of subsets of $X$ together with a function $k \colon X \times I \times I \imp I$ such that the following properties hold.
\begin{itemize}
\item If $x \in X$, then $x \in U_i$ for some $i \in I$.

\smallskip

\item If $x \in U_i \cap U_j$, then $x \in U_{k(x,i,j)} \subseteq U_i \cap U_j$.
\end{itemize}
\end{Definition}

\begin{Definition}[$\rca$; {\cite[Definition~2.2]{Dorais}}]
\label{def-CSCspace}
A \emph{strong countable second-countable space} (or simply \emph{countable second-countable space}) is a triple $(X, \MP U, k)$ where $\MP U = (U_i)_{i \in I}$ and $k \colon X \times I \times I \imp I$ form a base for a topology on the set $X$.
\end{Definition}

We have no use for the empty space, so we always assume that a countable second-countable space is non-empty.

Dorais also defines the notion of a \emph{weak base} for a topology on a set $X$ and the corresponding notion of a \emph{weak countable second-countable space}~\cite{DoraisPC}.  The distinction is that a weak base for a topology on $X$ is a uniformly enumerable sequence of subsets of $X$ rather than a sequence of literal subsets of $X$.  So in a weak base, membership in a basic open set is a $\Sigma^0_1$ property, whereas in a strong base, membership in a basic open set is a $\Delta^0_1$ property.  It is natural and straightforward to define a strong base for the order topology of a countable linear order in $\rca$ (see Definition~\ref{def-OrderedSpace} below), so in this work we need only consider strong bases and strong countable second-countable spaces.

Open subsets of countable second-countable spaces are coded by enumerations of indices of basic open sets.  Thus we must first define coded enumerable sets.

\begin{Definition}[$\rca$]\label{def-enum}
Let $A \subseteq \Nb$, and let $\Pf(A)$ denote the set of finite subsets of $A$.
\begin{itemize}
\item An \emph{enumerable} subset of $A$ is coded by a function $h \colon \Nb \imp \Pf(A)$, where $h$ codes $A_h = \bigcup_{n \in \Nb}h(n)$.

\smallskip

\item A sequence of \emph{uniformly enumerable} subsets of $A$ is coded by a function $h \colon \Nb \times \Nb \imp \Pf(A)$, where, for each $m$, $h(m, \cdot)$ codes the $m$\textsuperscript{th} set in the sequence: $A_{g(m, \cdot)} = \bigcup_{n \in \Nb}h(m,n)$.  Denote this sequence by $(A_{h(m, \cdot)} : m \in \Nb)$.
\end{itemize}
\end{Definition}

In general, $\aca$ is required to prove that $A_h = \bigcup_{n \in \Nb}h(n)$ exists as a set for every $A \subseteq \Nb$ and every $h \colon \Nb \imp \Pf(A)$.  Thus the expressions `$a \in A_h$' and `$a \in \bigcup_{n \in \Nb}h(n)$' must be interpreted as abbreviations for the formula `$\exists n (a \in h(n))$.'  The reason we consider functions $h \colon \Nb \imp \Pf(A)$ rather than functions $h \colon \Nb \imp A$ is that with functions $h \colon \Nb \imp \Pf(A)$, we may easily represent $\emptyset$ by the function with constant value $\emptyset$.

\begin{Definition}[$\rca$; {\cite[Definitions~2.3 and~2.4]{Dorais}}]\label{def-open}
Let $(X, \MP U, k)$ be a countable second-countable space, where $\MP U = (U_i)_{i \in I}$.  An \emph{effectively open} subset of $X$ is coded by an enumerable subset of $I$, i.e., by a function $h \colon \Nb \imp \Pf(I)$.  An $x \in X$ is a member of the effectively open subset of $X$ coded by $h$ if there are an $n \in \Nb$ and an $i \in h(n)$ such that $x \in U_i$.  Let $G_h = \bigcup_{n \in \Nb}\bigcup_{i \in h(n)}U_i$ denote the open subset of $X$ coded by $h$. 
\end{Definition}

Again, $\aca$ is required to show that $G_h$ exists as a set for every countable second-countable space $(X, \MP U, k)$ and function $h \colon \Nb \imp \Pf(I)$.  Thus the expression `$x \in G_h$' must be interpreted as an abbreviation for the formula `$(\exists n)(\exists i \in h(n))(x \in U_i)$.'

Let $(X, \MP U, k)$ be a countable second-countable space.  Every open subset of $X$ is an enumerable subset of $X$, meaning that for every $h \colon \Nb \imp \Pf(I)$, there is an $\hat{h} \colon \Nb \imp \Pf(X)$ such that $X_{\hat{h}} = G_h$:
\begin{align*}
\forall x [(\exists n)(x \in \hat{h}(n)) \biimp (\exists n)(\exists i \in h(n))(x \in U_i)].
\end{align*}
Furthermore, we may interpret any double-sequence $h \colon \Nb \times \Nb \imp \Pf(I)$ as a sequence $(G_{h(m,\cdot)} : m \in \Nb)$ of open sets, where the $m$\textsuperscript{th} open set in the sequence is $G_{h(m, \cdot)} = \bigcup_{n \in \Nb}\bigcup_{i \in h(m,n)}U_i$.  Each such sequence may also be thought of as a uniformly enumerable sequence of subsets of $X$.  That is, there is an $\hat{h} \colon \Nb \times \Nb \imp \Pf(X)$ such that, for all $m \in \Nb$, $X_{\hat{h}(m,\cdot)} = G_{h(m,\cdot)}$.  However, there may be sequences $(X_{\ell(m,\cdot)} : m \in \Nb)$ of uniformly enumerable subsets of $X$ where each individual $X_{\ell(m,\cdot)}$ is open, but there is no uniform way to code each $X_{\ell(m,\cdot)}$ as a union of basic open sets.  That is, it could be that for every $m$ there is an $h \colon \Nb \imp \Pf(I)$ such that $X_{\ell(m,\cdot)} = G_h$, but there is no $h \colon \Nb \times \Nb \imp \Pf(I)$ such that for every $m$, $X_{\ell(m,\cdot)} = G_{h(m,\cdot)}$.  We call a sequence of open subsets of $X$ \emph{honest} if each set in the sequence is uniformly coded as a union of basic open sets.

\begin{Definition}[$\rca$]
Let $(X, \MP U, k)$ be a countable second-countable space, where $\MP U = (U_i)_{i \in I}$.  A sequence $(X_{\ell(m,\cdot)} : m \in \Nb)$ of uniformly enumerable subsets of $X$ is an \emph{honest} sequence of open subsets of $X$ if there is a function $h \colon \Nb \times \Nb \imp \Pf(I)$ such that for all $m$, $X_{\ell(m,\cdot)} = G_{h(m, \cdot)}$.
\end{Definition}

Dorais gives two notions of compactness for countable second-countable spaces, corresponding to whether or not we require open covers to be honest.  We warn the reader that Dorais's original definition of compactness~\cite[Definition~3.1]{Dorais} considers only honest open covers, so the definition of `compact' in~\cites{Dorais, NoetherianSpaces} corresponds to the definition of `compact with respect to honest open covers' here.

\begin{Definition}[$\rca$]\label{def-comp}
Let $(X, \MP U, k)$ be a countable second-countable space.
\begin{itemize}
\item A sequence $(X_{h(m,\cdot)} : m \in \Nb)$ of uniformly enumerable subsets of $X$ is an \emph{open cover} of $X$ if $X_{h(m,\cdot)}$ is an open subset of $X$ for each $m$ and if $X = \bigcup_{m \in \Nb}X_{h(m,\cdot)}$ (i.e., $(\forall x \in X)(\exists m)(x \in X_{h(m,\cdot)})$).

\smallskip

\item An open cover $(X_{h(m,\cdot)} : m \in \Nb)$ of $X$ is \emph{honest} if $(X_{h(m,\cdot)} : m \in \Nb)$ is an honest sequence of open subsets of $X$.

\smallskip

\item The space $(X, \MP U, k)$ is \emph{compact} if for every open cover $(X_{h(m,\cdot)} : m \in \Nb)$ of $X$ there is an $M \in \Nb$ such that $X = \bigcup_{m < M}X_{h(m,\cdot)}$.

\smallskip

\item The space $(X, \MP U, k)$ is \emph{compact with respect to honest open covers} if for every honest open cover $(X_{h(m,\cdot)} : m \in \Nb)$ of $X$ there is an $M \in \Nb$ such that $X = \bigcup_{m < M}X_{h(m,\cdot)}$.
\end{itemize}
\end{Definition}

In terms of basic open sets, an honest open cover of $(X, \MP U, k)$ is an open cover of the form $X = \bigcup_{m \in \Nb}\bigcup_{n \in \Nb}\bigcup_{i \in h(m,n)}U_i$ for a function $h \colon \Nb \times \Nb \imp \Pf(I)$.  Thus $(X, \MP U, k)$ is compact w.r.t.\ honest open covers if and only if whenever $h \colon \Nb \times \Nb \imp \Pf(I)$ is such that $X = \bigcup_{m \in \Nb}\bigcup_{n \in \Nb}\bigcup_{i \in h(m,n)}U_i$, there is an $M \in \Nb$ such that $X = \bigcup_{m < M}\bigcup_{n \in \Nb}\bigcup_{i \in h(m,n)}U_i$.  A double-sequence of basic open sets may be rewritten as a single sequence of basic open sets, which means that a countable second-countable space is compact w.r.t.\ honest open covers if and only if it is compact w.r.t.\ honest open covers by basic open sets.  

\begin{Proposition}[$\rca$]
Let $(X, \MP U, k)$ be a countable second-countable space with $\MP U = (U_i)_{i \in I}$.  Then $(X, \MP U, k)$ is compact w.r.t.\ honest open covers if and only if for every $g \colon \Nb \imp I$ such that $X = \bigcup_{m \in \Nb}U_{g(m)}$, there is an $M \in \Nb$ such that $X = \bigcup_{m < M}U_{g(m)}$.  
\end{Proposition}

\begin{proof}
Every open cover of the form $X = \bigcup_{m \in \Nb}U_{g(m)}$ is an honest open cover.  So if $(X, \MP U, k)$ is compact w.r.t.\ honest open covers, then for every $g \colon \Nb \imp I$ such that $X = \bigcup_{m \in \Nb}U_{g(m)}$, there is an $M \in \Nb$ such that $X = \bigcup_{m < M}U_{g(m)}$.

For the converse, suppose that $X = \bigcup_{m \in \Nb}\bigcup_{n \in \Nb}\bigcup_{i \in h(m,n)}U_i$ for some function $h \colon \Nb \times \Nb \imp \Pf(I)$.  Fix $m_0$ and $n_0$ such that $h(m_0, n_0) \neq \emptyset$, and fix an $i_0 \in h(m_0, n_0)$.  For the purposes of this argument, let $\la \cdot, \cdot, \cdot \ra \colon \Nb^3 \imp \Nb$ denote a bijection.  Define $g \colon \Nb \imp I$ by
\begin{align*}
g(\la m, n, s \ra) =
\begin{cases}
\text{the $(s+1)$\textsuperscript{th} smallest member of $h(m,n)$} & \text{if $|h(m,n)| \geq s+1$}\\
i_0 & \text{otherwise}. 
\end{cases}
\end{align*}
Then $X = \bigcup_{p \in \Nb}U_{g(p)}$, so there is a $P \in \Nb$ such that $X = \bigcup_{p < P}U_{g(p)}$.  For every $p < P$ there are $m$ and $n$ such that $g(p) \in h(m,n)$.  By $\bso$, there is an $M \in \Nb$ such that for every $p < P$, there is an $m < M$ and an $n$ such that $g(p) \in h(m,n)$.  Therefore $X = \bigcup_{m < M}\bigcup_{n \in \Nb}\bigcup_{i \in h(m,n)}U_i$.
\end{proof}

In light of the above proposition, we typically think of compactness w.r.t.\ honest open covers in terms of covers of the form $X = \bigcup_{m \in \Nb}U_{g(m)}$ for functions $g \colon \Nb \imp I$.  Equivalently, we may also think of compactness w.r.t.\ honest open covers in terms of covers of the form $X = \bigcup_{n \in \Nb}\bigcup_{i \in h(n)}U_i$ for functions $h \colon \Nb \imp \Pf(I)$, as was (implicitly) done in Dorais's original definition~\cite[Definition~3.1]{Dorais} and in~\cite{NoetherianSpaces}.

When working with compactness in $\rca$, there is the added wrinkle that it may or may not be possible to uniformly determine whether or not a given finite collection of basic open sets covers the whole space.   If it is possible, then we say that the space's base has a \emph{finite cover relation}.

\begin{Definition}[$\rca$; {\cite[Definition~2.13]{Dorais}}]
Let $(X, \MP U, k)$ be a countable second-countable space with $\MP U = (U_i)_{i \in I}$.  $\MP U$ has a \emph{finite cover relation} if there is a set $C \subseteq \Pf(I)$ such that, for all $\{i_0, \dots, i_{n-1}\} \subseteq I$, $\{i_0, \dots, i_{n-1}\} \in C$ if and only if $X = \bigcup_{j < n} U_{i_j}$.
\end{Definition}

Every honest open cover of a countable second-countable space is also an open cover of the space, so $\rca$ proves that a compact countable second-countable space is also compact w.r.t.\ honest open covers.  Unsurprisingly, $\aca$ is required to prove that every countable second-countable space that is compact w.r.t.\ honest open covers is compact.  This fact follows from~\cite[Example~5.4]{Dorais}, but we find it instructive to present a similar yet somewhat more straightforward proof.  It is convenient to first introduce notions of discreteness.

\begin{Definition}[$\rca$; {\cite[Definition~5.1]{Dorais}}]
Let $(X, \MP U, k)$ be a countable second-countable space with $\MP U = (U_i)_{i \in I}$.  
\begin{itemize}
\item $(X, \MP U, k)$ is \emph{discrete} if for every $x \in X$ there is an $i \in I$ such that $U_i = \{x\}$.

\smallskip

\item $(X, \MP U, k)$ is \emph{effectively discrete} if there is a function $d \colon X \imp I$ such that, for every $x \in X$, $U_{d(x)} = \{x\}$.
\end{itemize}
\end{Definition}

We readily see that an infinite discrete countable second-countable space is not compact and that an infinite effectively discrete countable second-countable space is not compact w.r.t.\ honest open covers.

\begin{Proposition}\label{prop-CompImpHcomp}
The statement ``every countable second-countable space that is compact w.r.t.\ honest open covers is compact'' is equivalent to $\aca$ over $\rca$.
\end{Proposition}

\begin{proof}
For the forward direction, when working in $\aca$ it is routine to show that every sequence of open subsets of a countable second-countable space is honest.  Thus $\aca$ proves that compactness and compactness w.r.t.\ honest open covers are equivalent.

For the reverse direction, let $f \colon \Nb \imp \Nb$ be an injection.  We appeal to Lemma~\ref{lem-ACAinjection} and show that the range of $f$ exists.  Define a countable second-countable space $(X, \MP U, k)$ by $X = \Nb$, $I = \{0,1,2\} \times \Nb$, and
\begin{align*}
U_{\la 0, n \ra} &= \{n\} \cup \{t : (\exists s \leq t)(f(s) = n)\}\\
U_{\la 1, \la n, s \ra \ra} &=
\begin{cases}
\{n\} & \text{if $f(s) = n$}\\
\emptyset & \text{if $f(s) \neq n$}
\end{cases}\\
U_{\la 2, s \ra} &= \{t : t \geq s\}.
\end{align*}

The function $k$ is computed as follows.  To check that $k$ behaves as intended, it is helpful to observe that if $x \in U_{\la 0, n \ra}$ but $x \neq n$, then $U_{\la 2, x \ra} \subseteq U_{\la 0, n \ra}$.
\begin{itemize}
\item For $k(x, \la 0, m \ra, \la 0, n \ra)$:
\begin{itemize}
\item If $m = n$, then output $\la 0, m \ra$.
\item If $m \neq n$ and either $x = m$ or $x = n$, then check if there is an $s \leq x$ such that $f(s) = x$.  If so, output $\la 2, x \ra$.  If not, output $\la 0, x \ra$.
\item If $m \neq n$, $x \neq m$, and $x \neq n$, then output $\la 2, x \ra$.
\end{itemize}

\smallskip

\item For $k(x, \la 0, m \ra, \la 1, \la n, s \ra \ra)$ and $k(x, \la 1, \la n, s \ra \ra, \la 0, m \ra)$:  Output $\la 1, \la n, s \ra \ra$.

\smallskip

\item For $k(x, \la 0, m \ra, \la 2, s \ra)$ and $k(x, \la 2, s \ra, \la 0, m \ra)$:
\begin{itemize}
\item If $x = m$, check if there is a $t \leq x$ such that $f(t) = x$.  If so, output $\la 2, x \ra$.  If not, output $\la 0, x \ra$.
\item If $x \neq m$, output $\la 2, x \ra$.
\end{itemize}

\smallskip

\item For $k(x, \la 1, \la m, s \ra \ra, \la 1, \la n, t \ra \ra)$, output $\la 1, \la m, s \ra \ra$.

\smallskip

\item For $k(x, \la 1, \la m, s \ra \ra, \la 2, t \ra)$ and $k(x, \la 2, t \ra, \la 1, \la m, s \ra \ra)$, output $\la 1, \la m, s \ra \ra$.

\smallskip

\item For $k(x, \la 2, s \ra, \la 2, t \ra)$, output $\la 2, \max\{s,t\} \ra$.

\end{itemize}

The space $(X, \MP U, k)$ is discrete because for every $n$, either $U_{\la 0, n \ra} = \{n\}$ or there is an $s$ such that $U_{\la 1, \la n, s \ra \ra} = \{n\}$.  Thus $(X, \MP U, k)$ is not compact, and therefore it is not compact w.r.t.\ honest open covers.  Let $h \colon \Nb \imp I$ be such that $X = \bigcup_{m \in \Nb}U_{h(m)}$, but such that there is no $M \in \Nb$ for which $X = \bigcup_{m < M}U_{h(m)}$.  Notice that every basic open set is either finite or cofinite.  If $U_{h(m)}$ is cofinite for some $m$, then by using $\bso$ and the assumption $X = \bigcup_{m \in \Nb}U_{h(m)}$, we may conclude that there is an $M \in \Nb$ such that $X = \bigcup_{m < M}U_{h(m)}$, which is a contradiction.  Thus $U_{h(m)}$ is finite for every $m$.  Consider an $n \in X$.  There must be an $m$ such that $n \in U_{h(m)}$.  If there is an $s$ such that $f(s) = n$, then the only finite basic open set that contains $n$ is $U_{\la 1, \la n, s \ra \ra}$, so in this case it must be that $h(m) = \la 1, \la n, s \ra \ra$.  If instead there is no $s$ such that $f(s) = n$, then the only finite basic open set that contains $n$ is $U_{\la 0, n \ra}$, so in this case it must be that $h(m) = \la 0, n \ra$.  Therefore
\begin{align*}
\exists s(f(s) = n) \Biimp \exists m \exists s(h(m) = \la 1, \la n, s \ra \ra) \Biimp \forall m(h(m) \neq \la 0, n \ra).
\end{align*}
Thus the range of $f$ exists by $\Delta^0_1$-comprehension.
\end{proof}

Dorais's~\cite[Example~5.4]{Dorais} shows that, over $\rca$, $\aca$ is equivalent to the statement ``every infinite countable second-countable space that is compact w.r.t.\ honest open covers and whose base has a finite cover relation is not discrete.''  The proof of Proposition~\ref{prop-CompImpHcomp} may also be seen as a proof of this fact, as one may check that the constructed $\MP U$ has a finite cover relation.

This work concerns the order topologies of countable linear orders, which give natural examples of (strong) countable second-countable spaces.

\begin{Definition}[$\rca$; {\cite[Definition~7.1]{Dorais}}]~\label{def-OrderedSpace}
Let $(L, \prec)$ be a linear order.  The base for the \emph{order topology} on $L$ is given by $\MP U = (U_i)_{i \in I}$ and $k : L \times I \times I \imp I$, where
\begin{itemize}
\item $I = (L \cup\{-\infty, \infty\}) \times (L \cup\{-\infty, \infty\})$,

\smallskip

\item $U_{\la a, b \ra} = (a, b) = \{x \in L : a \prec x \prec b\}$, for $\la a,b \ra \in I$, and

\smallskip

\item $k(x, \la a_0, b_0 \ra, \la a_1, b_1 \ra) = \la \max(a_0, a_1), \min(b_0, b_1) \ra$ for $x \in L$ and $\la a_0, b_0 \ra, \la a_1, b_1 \ra \in I$.
\end{itemize}
The \emph{ordered space} associated with $L$ is the countable second-countable space $(L, \MP U, k)$.
\end{Definition}
In the above definition, $-\infty$ and $\infty$ are (codes for) two distinct fresh symbols not in $L$.  We extend $\prec$ to $L \cup\{-\infty, \infty\}$ by setting $-\infty \prec x \prec \infty$ for all $x \in L$.  Note that, as a matter of convenience, we allow $\la a, b \ra \in I$ even when $b \prec a$, in which case $U_{\la a, b \ra} = \emptyset$.  As the basic open subsets of $L$ are particularly easy to describe, we dispense with the notational encumbrances of Definition~\ref{def-open} and simply write an enumeration of basic open sets as $((a_n, b_n) : n \in \Nb)$, with the understanding that $a_n, b_n \in L \cup \{-\infty, \infty\}$ for each $n$.  Notice that the base of an ordered space always has a finite cover relation~\cite[Proposition~7.5]{Dorais}.

As mentioned above, Dorais observes that $\rca$ proves that if the order topology of $(L, \prec)$ is compact w.r.t.\ honest open covers, then $(L, \prec)$ is complete.

\begin{Lemma}[{\cite[Section~7]{Dorais}}]\label{lem-CompactImpComplete}
$\rca$ proves the statement ``for every countable linear order $(L, \prec)$, if the order topology of $(L, \prec)$ is compact w.r.t.\ honest open covers, then $(L, \prec)$ is complete.''  It follows that $\rca$ also proves the statement ``for every countable linear order $(L, \prec)$, if the order topology of $(L, \prec)$ is compact, then $(L, \prec)$ is complete.''
\end{Lemma}

\begin{proof}
We prove the contrapositive.  Suppose that $(L, \prec)$ is not complete, and let $A^- \cup A^+ = L$ be a partition where $(\forall x \in A^-)(\forall y \in A^+)(x \prec y)$, but is such that $A^-$ has no maximum element and $A^+$ has no minimum element.  Then any enumeration of the basic open sets of the form $(-\infty, b)$ for $b \in A^-$ and $(a, \infty)$ for $a \in A^+$ is an honest open cover of $L$ by basic open sets that has no finite subcover.  Thus the order topology of $(L, \prec)$ is not compact w.r.t.\ honest open covers.
\end{proof}

We show the following in the next section.
\begin{itemize}
\item The statement ``for every countable linear order $(L, \prec)$, if $(L, \prec)$ is complete, then the order topology of $(L, \prec)$ is compact w.r.t.\ honest open covers'' is equivalent to $\wkl$ over $\rca$.

\smallskip

\item The statement ``for every countable linear order $(L, \prec)$, if $(L, \prec)$ is complete, then the order topology of $(L, \prec)$ is compact'' is equivalent to $\aca$ over $\rca$.
\end{itemize}

\section{The strength of compactness for complete linear orders}

First, we show that $\wkl$ proves that the order topology of a complete linear order is compact w.r.t.\ honest open covers.  It follows that $\aca$ proves that the order topology of a complete linear order is compact.  The proof is essentially an implementation of the usual argument as found, for example, in the proof of~\cite[Theorem~27.1]{Munkres}.

\begin{Lemma}\label{lem-fwWKL}
$\wkl$ proves the statement ``for every countable linear order $(L, \prec)$, if $(L, \prec)$ is complete, then the order topology of $(L, \prec)$ is compact w.r.t.\ honest open covers.''
\end{Lemma}

\begin{proof}
We prove the contrapositive of the statement in $\wkl$.  Suppose that the order topology of $(L, \prec)$ is not compact w.r.t.\ honest open covers, and let $((a_n, b_n) : n \in \Nb)$ be an open cover of $L$ by basic open sets with no finite subcover.  Assume that $0$ is the minimum element of $L$ and that $1$ is the maximum element of $L$, for if either $L$ has no minimum or $L$ has no maximum, then $(L, \prec)$ is not complete, as desired.

Define a \emph{linkage} in $L$ to be a finite sequence $((a_{n_i}, b_{n_i}) : i < k)$ of intervals from the cover such that $(\forall i < k-1)(a_{n_{i+1}} \prec b_{n_i} \prec b_{n_{i+1}})$.  Say that an $\ell \in L$ is in a linkage $((a_{n_i}, b_{n_i}) : i < k)$ if $(\exists i < k)(a_{n_i} \prec \ell \prec b_{n_i})$ (i.e., if $\ell \in \bigcup_{i < k}(a_{n_i}, b_{n_i})$).  Notice that no linkage contains both $0$ and $1$ because such a linkage would be a finite subcover of $((a_n, b_n) : n \in \Nb)$.  A straightforward application of $\Pi^0_1$ induction on the length of a linkage shows that the union of a linkage is an interval, meaning that if $((a_{n_i}, b_{n_i}) : i < k)$ is a linkage, if $x$ and $y$ are both in the linkage, and if $x \prec y$, then every $z \in (x, y)$ is also in the linkage.

We define the tree $T \subseteq 2^{<\Nb}$ consisting of all sequences $\sigma$ that look like initial segments of sets that contain $0$, do not contain $1$, are $\prec$-downward-closed, and are closed under linkages.  Let $T$ be the set of all $\sigma \in 2^{<\Nb}$ satisfying the following conditions.
\begin{itemize}
\item If $|\sigma| > 0$, then $\sigma(0) = 1$.

\smallskip

\item If $|\sigma| > 1$, then $\sigma(1) = 0$.

\smallskip

\item For all $x < |\sigma|$, if $x \notin L$, then $\sigma(x) = 0$.

\smallskip

\item For all $x, y < |\sigma|$, if $x, y \in L$, $x \prec y$, and $\sigma(y) = 1$, then $\sigma(x) = 1$.

\smallskip

\item For all $x, y < |\sigma|$ and for all finite sequences $\la n_i : i < k \ra < |\sigma|$, if $x, y \in L$, if $((a_{n_i}, b_{n_i}) : i < k)$ is a linkage containing both $x$ and $y$, and if $\sigma(y) = 1$, then $\sigma(x) = 1$.
\end{itemize}
$T$ is $\Delta^0_1$ relative to $(L, \prec)$ and hence exists by $\Delta^0_1$ comprehension.  It is easy to see that $T$ is closed under initial segments and hence is a tree.  We show that $T$ is infinite.  To this end, let $n \in \Nb$, and let
\begin{align*}
D = \{x < n : (\exists \la n_i : i < k \ra)[\text{$((a_{n_i}, b_{n_i}) : i < k)$ is a linkage containing both $0$ and $x$}]\}.
\end{align*}
The set $D$ exists by bounded $\Sigma^0_1$ comprehension.  Let $\sigma \in 2^{<\Nb}$ be the sequence of length $n$ where, for all $x < n$, $\sigma(x) = 1$ if $x \in D$ and $\sigma(x) = 0$ if $x \notin D$.  Then $\sigma \in T$, which can be seen by noticing that $0 \in D$ (if $n > 0$); that $1 \notin D$ because $0$ and $1$ are not in a linkage together; and that $D$ is $\prec$-downwards closed in $L \cap \{0, 1, \dots, n-1\}$ because the union of any linkage containing $0$ is an initial segment of $L$.  Therefore, for every $n$ there is a $\sigma \in T$ with $|\sigma| = n$.  Hence $T$ is infinite.

Thus $T$ is an infinite subtree of $2^{<\Nb}$.  Apply weak K\"onig's lemma to $T$ to get an infinite path, and view that path as the characteristic function of a set $X \subseteq L$.  $X$ contains $0$, does not contain $1$, is $\prec$-downward-closed, and is closed under linkages.  By setting $A^- = X$ and $A^+ = L \setminus X$, we obtain a partition $L = A^- \cup A^+$ where $0 \in A^-$, $1 \in A^+$, and $(\forall x \in A^-)(\forall y \in A^+)(x \prec y)$.

We show that $A^-$ has no maximum element and that $A^+$ has no minimum element.  First, suppose for a contradiction that $A^-$ has a maximum element $\ell$.  As $((a_n, b_n) : n \in \Nb)$ is a cover, let $(a_{n_0}, b_{n_0})$ be such that $\ell \in (a_{n_0}, b_{n_0})$.  We cannot have that $b_{n_0} = \infty$, for otherwise $\ell$ and $1$ are both in the linkage $(a_{n_0}, b_{n_0})$, which implies that $1 \in A^-$ because $A^-$ is linkage-closed.  Thus $b_{n_0} \in L$, so let $(a_{n_1}, b_{n_1})$ be such that $b_{n_0} \in (a_{n_1}, b_{n_1})$.  Then $(a_{n_0}, b_{n_0}), (a_{n_1}, b_{n_1})$ is a linkage containing both $\ell$ and $b_{n_0}$.  Thus $b_{n_0} \succ \ell$ is in $A^-$ because $A^-$ is linkage-closed.  This contradicts that $\ell$ is the maximum of $A^-$.  Now suppose for a contradiction that $A^+$ has a minimum element $\ell$.  Again by the fact that $((a_n, b_n) : n \in \Nb)$ is a cover, let $(a_{m_1}, b_{m_1})$ be such that $\ell \in (a_{m_1}, b_{m_1})$.  We cannot have that $a_{m_1} = -\infty$, for otherwise $0$ and $\ell$ are both in the linkage $(a_{m_1}, b_{m_1})$, which implies that $\ell \in A^-$ because $0 \in A^-$ and $A^-$ is linkage-closed.  Thus $a_{m_1} \in L$, so let $(a_{m_0}, b_{m_0})$ be such that $a_{m_1} \in (a_{m_0}, b_{m_0})$.  Then either $b_{m_0} \prec b_{m_1}$, in which case $(a_{m_0}, b_{m_0}), (a_{m_1}, b_{m_1})$ is a linkage containing both $\ell$ and $a_{m_1}$; or $b_{m_0} \succeq b_{m_1}$, in which case $(a_{m_0}, b_{m_0}) \supseteq (a_{m_1}, b_{m_1})$ is a linkage containing both $\ell$ and $a_{m_1}$.  Thus $a_{m_1}$ is in a linkage with $\ell$.  However, $a_{m_1} \in A^-$ because $a_{m_1} \prec \ell$ and $\ell$ is the minimum element of $A^+$.  This contradicts that $A^-$ is linkage-closed.  Thus $A^-$ has no maximum element, and $A^+$ has no minimum element.  So $A^-$ and $A^+$ witness that $(L, \prec)$ is not complete.
\end{proof}

\begin{Corollary}\label{cor-fwACA}
$\aca$ proves the statement ``for every countable linear order $(L, \prec)$, if $(L, \prec)$ is complete, then the order topology of $(L, \prec)$ is compact.''
\end{Corollary}

\begin{proof}
$\aca$ proves $\wkl$ and that a countable second-countable space is compact if and only if it is compact w.r.t.\ honest open covers.
\end{proof}

We now give the reversals.  The strategy is as follows.  First, recall the \emph{Kleene-Brouwer ordering} of finite sequences:  $\sigma \leqKB \tau$ if either $\sigma$ is an extension of $\tau$ or $\sigma$ is to the left of $\tau$.  That is, $\sigma \leqKB \tau$ if
\begin{align*}
\sigma \supseteq \tau \orr (\exists n < \min(|\sigma|, |\tau|))[\sigma(n) < \tau(n) \andd (\forall i < n)(\sigma(i) = \tau(i))].
\end{align*}
Now, let $T \subseteq \Nb^{<\Nb}$ be an infinite finitely-branching tree.  In the case of $\aca$, we show that the order topology of $(T, \leKB)$ is discrete, hence not compact.  In the case of $\wkl$, we additionally assume that $T$ is bounded, and we show that the order topology of $(T, \leKB)$ is effectively discrete, hence not compact w.r.t.\ honest open covers.  In both cases, we conclude that $(T, \leKB)$ is not complete, which lets us extract an infinite path through $T$ from a witnessing partition.  The idea of last step of this this strategy, to use a certain partition of a linear order on $T$ to find a path through $T$, also appears in Simpson and Yokoyama's analysis of Peano categoricity~\cite{SimpsonYokoyama}.

\begin{Lemma}[$\rca$]\label{lem-discreteKB}
Let $T \subseteq \Nb^{<\Nb}$ be an infinite, finitely-branching tree, and for each $\sigma \in T$, let $T_{\sigma} = \{\tau \in T : \tau \supseteq \sigma\}$ denote the full subtree of $T$ above $\sigma$.  Assume that every $T_\sigma$ has a $\leKB$-least element.  Then the order topology of $(T, \leKB)$ is discrete.  Moreover, if $T$ is bounded, then the order topology of $(T, \leKB)$ is effectively discrete.
\end{Lemma}

\begin{proof}
The lemma follows from two claims.

\begin{Claim}
If $\sigma \in T$ is not $\leKB$-least, then $\sigma$ has a $\leKB$-immediate predecessor.  If $T$ is bounded, then the $\leKB$-immediate predecessor can be found effectively.  That is, if $T$ is bounded, then there is a function $\ell \colon T \imp T \cup \{-\infty\}$ such that
\begin{align*}
\ell(\sigma) =
\begin{cases}
\textup{$\sigma$'s $\leKB$-immediate predecessor} & \textup{if $\sigma$ is not $\leKB$-least}\\
-\infty & \textup{if $\sigma$ is $\leKB$-least}.
\end{cases} 
\end{align*}
\end{Claim}

\begin{proof}[Proof of claim]
Consider a $\sigma \in T$ that is not $\leKB$-least.  If $\sigma$ is not a leaf, then its rightmost child is its $\leKB$-immediate predecessor.  In this case, the rightmost child exists because $T$ is finitely-branching.

Suppose that $\sigma$ is a leaf.  Let $i < |\sigma|$ be greatest such that for some $m < \sigma(i)$, $(\sigma \rst i)^\smf m \in T$.  Such an $i$ exists by the assumption that $\sigma$ is not $\leKB$-least (and that $\sigma$ is a leaf).  Given this greatest $i$, let $m < \sigma(i)$ be greatest such that $(\sigma \rst i)^\smf m \in T$.  Then this $\tau = (\sigma \rst i)^\smf m$ is $\sigma$'s $\leKB$-immediate predecessor.  To see this, suppose that $\alpha \leKB \sigma$ for some $\alpha \in T$.  Let $j$ be such that $\alpha \rst j = \sigma \rst j$ and $\alpha(j) < \sigma(j)$.  Then $j \leq i$ by the maximality of $i$.  If $j = i$, then $\alpha \leqKB \tau$ by the maximality of $m$.  If $j < i$, then $\alpha \leKB \tau$ because $\alpha \rst j = \sigma \rst j = \tau \rst j$ but $\alpha(j) < \sigma(j) = \tau(j)$.

If $T$ is bounded by $f$, then $f$ can be used to determine whether or not $\sigma$ is a leaf and, if not, determine $\sigma$'s rightmost child.  If $\sigma$ is a leaf, no further use of $f$ is required to produce the $\leKB$-immediate predecessor $\tau$ because in this case $\sigma$ itself provides the necessary bounds.
\end{proof}

\begin{Claim}
If $\sigma \in T$ is not $\leKB$-greatest (i.e., if $\sigma \neq \emptyset$), then $\sigma$ has a $\leKB$-immediate successor.  If $T$ is bounded, then the $\leKB$-immediate successor can be found effectively.  That is, if $T$ is bounded, then there is a function $r \colon T \imp T \cup \{\infty\}$ such that
\begin{align*}
r(\sigma) =
\begin{cases}
\textup{$\sigma$'s $\leKB$-immediate successor} & \textup{if $\sigma \neq \emptyset$}\\
\infty & \textup{if $\sigma = \emptyset$}.
\end{cases} 
\end{align*}

\end{Claim}

\begin{proof}[Proof of claim]
Here we use the \emph{ad hoc} notation `$\alpha \bang i$' to denote the sequence obtained by changing the last entry of $\alpha \neq \emptyset$ to $i$.  

Consider a $\sigma \in T$ that is not $\emptyset$.  If there is no $m > \sigma(|\sigma|-1)$ such that $\sigma \bang m \in T$, then $\sigma \rst (|\sigma|-1)$ is $\sigma$'s $\leKB$-immediate successor.

If there is an $m > \sigma(|\sigma|-1)$ such that $\sigma \bang m \in T$, then let $m$ be the least such $m$.  Then $\sigma$'s $\leKB$-immediate successor is the $\leKB$-least element $\tau$ of $T_{\sigma \bang m}$, which exists by assumption.  To see this, suppose that $\alpha \geKB \sigma$ for some $\alpha \in T$.  If $\alpha \subsetneq \sigma$, then $\alpha \subsetneq \sigma \bang m \subseteq \tau$, so $\alpha \geKB \tau$.  Otherwise, there is a $j$ such that $\alpha \rst j = \sigma \rst j$ and $\alpha(j) > \sigma(j)$.  If $j < |\sigma|-1$, then $\alpha \geKB \sigma \bang m \geqKB \tau$.  If $j = |\sigma|-1$, then either $\alpha(j) > m$ or $\alpha(j) = m$.  If $\alpha(j) > m$, then again $\alpha \geKB \sigma \bang m \geqKB \tau$.  If $\alpha(j) = m$, then $\alpha \in T_{\sigma \bang m}$, so $\alpha \geqKB \tau$ by the choice of $\tau$.

If $T$ is bounded by $f$, then $f$ can be used to determine whether or not there is an $m > \sigma(|\sigma|-1)$ with $\sigma \bang m \in T$.  Furthermore, $f$ can be used to find the $\leKB$-least element of any subtree $T_\eta$.  The $\leKB$-least element of $T_\eta$ is the leftmost leaf of $T_\eta$, which can be found by starting at $\eta$ and following the leftmost child until reaching a leaf.  The bound $f$ can be used to determine whether or not a given element of $T$ is a leaf, so this search is effective.
\end{proof}

Consider now the order topology of $(T, \leKB)$, and consider a $\sigma \in T$.  If $\sigma$ is neither $\leKB$-least nor $\leKB$-greatest, then $\sigma$ has a $\leKB$-immediate predecessor $\tau$ and a $\leKB$-immediate successor $\eta$.  In this case, $\{\sigma\} = (\tau, \eta)$ is a basic open set.  If $\sigma$ is $\leKB$-least, then $\sigma$ has a $\leKB$-immediate successor $\eta$.  In this case, $\{\sigma\} = (-\infty, \eta)$ is a basic open set.  If $\sigma$ is $\leKB$-greatest, then $\sigma$ has a $\leKB$-immediate predecessor $\tau$.  In this case $\{\sigma\} = (\tau, \infty)$ is a basic open set.  Thus the order topology of $(T, \leKB)$ is discrete.  Furthermore, if $T$ is bounded, then $\{\sigma\} = (\ell(\sigma), r(\sigma))$ for every $\sigma \in T$.  Thus the order topology of $(T, \leKB)$ is effectively discrete via the function $d(\sigma) = \la \ell(\sigma), r(\sigma) \ra$.
\end{proof}

\begin{Lemma}[$\rca$]\label{lem-ExtractPath}
Let $T \subseteq \Nb^{<\Nb}$ be an infinite finitely-branching tree.  If the linear order $(T, \leKB)$ is not complete, then $T$ has an infinite path.
\end{Lemma}

\begin{proof}
Let $A^-$ and $A^+$ witness that $(T, \leKB)$ is not complete.  Notice that $A^+$ is non-empty (because $\emptyset \in A^+$, as otherwise it would be the greatest element of $A^-$) and that $A^+$ is closed under initial segments (because it is $\leKB$-upward closed).  Define a set $X \subseteq \Nb^{<\Nb}$ by putting $\sigma \in X$ if and only if 
\begin{align*}
\sigma \in A^+ \andd (\forall i < |\sigma|)(\forall n < \sigma(i))[(\sigma \rst i)^\smf n \notin A^+].
\end{align*}
For every $n$, the set $X$ contains at most one sequence of length $n$, which can be seen from the definition of $X$ and the fact that $A^+$ is closed under initial segments.  Using $\iso$, we show that, for every $n$, $X$ contains at least one sequence of length $n$.  For the base case, $\emptyset$ is a sequence in $X$ of length $0$.  For the inductive case, suppose that $X$ contains a sequence $\sigma$ of length $n$.  Then $\sigma \in A^+$ and, as $A^+$ has no $\leKB$-least element, there must be a sequence $\tau \in A^+$ with $\tau \leKB \sigma$.  The sequence $\tau$ cannot be to the left of $\sigma$ because this would contradict $\sigma \in X$.  Therefore it must be that $\tau \supseteq \sigma$.  It follows that there is an $m$ such that $\sigma^\smf m \in A^+$.  Let $m$ be least such that $\sigma^\smf m \in A^+$.  Then $\sigma^\smf m$ is in $X$ and has length $n+1$.  This completes the induction.

The set $X$ thus contains exactly one sequence of each length.  As $X$ is closed under initial segments, it follows that if $\sigma, \tau \in X$ are such that $|\sigma| \leq |\tau|$, then $\sigma \subseteq \tau$.  Thus if we let $\sigma_i$ be the sequence in $X$ of length $i$, then $\sigma_0 \subseteq \sigma_1 \subseteq \sigma_2 \subseteq \cdots$ defines an infinite path through $T$.
\end{proof}

\begin{Theorem}{\ }
\begin{enumerate}[(i)]
\item\label{it-CompWKL} The statement ``for every countable linear order $(L, \prec)$, if $(L, \prec)$ is complete, then the order topology of $(L, \prec)$ is compact w.r.t.\ honest open covers'' is equivalent to $\wkl$ over $\rca$.

\smallskip

\item\label{it-CompACA} The statement ``for every countable linear order $(L, \prec)$, if $(L, \prec)$ is complete, then the order topology of $(L, \prec)$ is compact'' is equivalent to $\aca$ over $\rca$.
\end{enumerate}
\end{Theorem}

\begin{proof}
The forward direction of item~(\ref{it-CompWKL}) is Lemma~\ref{lem-fwWKL}, and the forward direction of item~(\ref{it-CompACA}) is Corollary~\ref{cor-fwACA}.

For the reversal of item~(\ref{it-CompACA}), let $T \subseteq \Nb^{< \Nb}$ be an infinite finitely-branching tree.  By Theorem~\ref{thm-KLequivs} item~(\ref{it-fullKL}), it suffices to show that $T$ has an infinite path.  Recall that, for $\sigma \in T$, $T_{\sigma} = \{\tau \in T : \tau \supseteq \sigma\}$ denotes the full subtree of $T$ above $\sigma$.  Suppose that there is a $\sigma \in T$ such that $T_\sigma$ has no $\leKB$-least element.  Then setting $A^- = \emptyset$ and $A^+ = T_\sigma$ gives a partition of $T_\sigma$ witnessing that the linear order $(T_\sigma, \leKB)$ is not complete.  By Lemma~\ref{lem-ExtractPath}, there is an infinite path through $T_\sigma$, which is an infinite path through $T$.  Suppose instead that $T_\sigma$ has a $\leKB$-least element for every $\sigma \in T$.  Then the order topology of $(T, \leKB)$ is discrete by Lemma~\ref{lem-discreteKB}.  Therefore the order topology of $(T, \leKB)$ is not compact.  We assume that the order topology of a complete linear order is compact, so $(T, \leKB)$ is not complete.  Therefore $T$ has an infinite path by Lemma~\ref{lem-ExtractPath}.

The reversal of item~(\ref{it-CompWKL}) is analogous.  Let $T \subseteq \Nb^{< \Nb}$ be an infinite bounded tree.  By Theorem~\ref{thm-KLequivs} item~(\ref{it-bddKL}), it suffices to show that $T$ has an infinite path.  As above, if there is a $\sigma \in T$ such that $T_\sigma$ has no $\leKB$-least element, then the linear order $(T_\sigma, \leKB)$ is not complete, so $T_\sigma$, and therefore $T$, has an infinite path by Lemma~\ref{lem-ExtractPath}.  If $T_\sigma$ has a $\leKB$-least element for every $\sigma \in T$, then the order topology of $(T, \leKB)$ is effectively discrete by Lemma~\ref{lem-discreteKB}.  Therefore the order topology of $(T, \leKB)$ is not compact w.r.t.\ honest open covers.  Therefore the linear order $(T, \leKB)$ is not complete.  Therefore $T$ has an infinite path by Lemma~\ref{lem-ExtractPath}.
\end{proof}

\section*{Acknowledgments}

We thank Fran\c{c}ois Dorais and Giovanni Sold\`a for helpful discussions.  This project was partially supported by a grant from the John Templeton Foundation (``A new dawn of intuitionism: mathematical and philosophical advances'' ID 60842).  The opinions expressed in this work are those of the author and do not necessarily reflect the views of the John Templeton Foundation.

\bibliography{ShaferCompactCompleteLO}

\begin{bibdiv}
\begin{biblist}

\bib{Dorais}{unpublished}{
      author={Dorais, Fran{\c{c}}ois~G.},
       title={{Reverse mathematics of compact countable second-countable
  spaces}},
        date={2011},
        note={arXiv:1110.6555v1},
}

\bib{DoraisPC}{misc}{
      author={Dorais, Fran{\c{c}}ois~G.},
       title={Personal communication},
        date={2018},
}

\bib{Friedman}{incollection}{
      author={Friedman, Harvey},
       title={Some systems of second order arithmetic and their use},
        date={1975},
   booktitle={Proceedings of the {I}nternational {C}ongress of {M}athematicians
  ({V}ancouver, {B}. {C}., 1974), {V}ol. 1},
   publisher={Canad. Math. Congress, Montreal, Que.},
       pages={235\ndash 242},
      review={\MR{0429508}},
}

\bib{NoetherianSpaces}{article}{
      author={Frittaion, Emanuele},
      author={Hendtlass, Matthew},
      author={Marcone, Alberto},
      author={Shafer, Paul},
      author={Van~der Meeren, Jeroen},
       title={Reverse mathematics, well-quasi-orders, and {N}oetherian spaces},
        date={2016},
        ISSN={0933-5846},
     journal={Archive for Mathematical Logic},
      volume={55},
      number={3-4},
       pages={431\ndash 459},
         url={https://doi.org/10.1007/s00153-015-0473-4},
      review={\MR{3490913}},
}

\bib{HajekPudlak}{book}{
      author={H\'{a}jek, Petr},
      author={Pudl\'{a}k, Pavel},
       title={{M}etamathematics of {F}irst-{O}rder {A}rithmetic},
      series={Perspectives in Mathematical Logic},
   publisher={Springer-Verlag, Berlin},
        date={1998},
        ISBN={3-540-63648-X},
        note={Second printing},
      review={\MR{1748522}},
}

\bib{MummertMF}{article}{
      author={Mummert, Carl},
       title={Reverse mathematics of {MF} spaces},
        date={2006},
        ISSN={0219-0613},
     journal={Journal of Mathematical Logic},
      volume={6},
      number={2},
       pages={203\ndash 232},
         url={https://doi.org/10.1142/S0219061306000578},
      review={\MR{2317427}},
}

\bib{MummertSimpson}{article}{
      author={Mummert, Carl},
      author={Simpson, Stephen~G.},
       title={Reverse mathematics and {$\Pi_2^1$} comprehension},
        date={2005},
        ISSN={1079-8986},
     journal={Bulletin of Symbolic Logic},
      volume={11},
      number={4},
       pages={526\ndash 533},
         url={http://projecteuclid.org/euclid.bsl/1130335208},
      review={\MR{2198712}},
}

\bib{Munkres}{book}{
      author={Munkres, James~R.},
       title={Topology},
     edition={Second Edition},
   publisher={Prentice Hall, Inc., Upper Saddle River, NJ},
        date={2000},
        ISBN={0-13-181629-2},
      review={\MR{3728284}},
}

\bib{SimpsonSOSOA}{book}{
      author={Simpson, Stephen~G.},
       title={{S}ubsystems of {S}econd {O}rder {A}rithmetic},
     edition={Second Edition},
      series={Perspectives in Logic},
   publisher={Cambridge University Press, Cambridge; Association for Symbolic
  Logic, Poughkeepsie, NY},
        date={2009},
        ISBN={978-0-521-88439-6},
         url={https://doi.org/10.1017/CBO9780511581007},
      review={\MR{2517689}},
}

\bib{SimpsonYokoyama}{article}{
      author={Simpson, Stephen~G.},
      author={Yokoyama, Keita},
       title={Reverse mathematics and {P}eano categoricity},
        date={2013},
        ISSN={0168-0072},
     journal={Annals of Pure and Applied Logic},
      volume={164},
      number={3},
       pages={284\ndash 293},
         url={https://doi.org/10.1016/j.apal.2012.10.014},
      review={\MR{3001547}},
}

\end{biblist}
\end{bibdiv}

\vfill

\end{document}